\documentclass[12pt]{amsart}
\usepackage{amsfonts,amsmath,amssymb,amscd,mathrsfs}
\usepackage[hidelinks]{hyperref}
\usepackage{xcolor}
\usepackage{soul}
\sethlcolor{cyan}
\usepackage{verbatim}
\usepackage{enumitem}
\usepackage{hyperref}
\usepackage{tikz-cd}

\numberwithin{equation}{section}

\newtheorem{theorem}{Theorem}[section]

\newtheorem{corollary}[theorem]{Corollary}
\newtheorem{proposition}[theorem]{Proposition}

\theoremstyle{definition}
\newtheorem{definition}[theorem]{Definition}
\newtheorem{example}[theorem]{Example}

\theoremstyle{remark}
\newtheorem{remark}[theorem]{Remark}

\def \D{\mathbb{D}}

\def \Z{\mathbb{Z}}

\newcommand{\cla}{\mathcal{A}}
\newcommand{\clb}{\mathcal{B}}

\newcommand{\cle}{\mathcal{E}}
\newcommand{\clh}{\mathcal{H}}

\newcommand{\clm}{\mathcal{M}}

\newcommand{\clk}{\mathcal{K}}

\newcommand{\clf}{\mathcal{F}}
\newcommand{\clq}{\mathcal{Q}}

\subjclass[2020]{Primary: 46L05, 46L80; Secondary: 47A13, 47A53, 46L55}
\keywords{Odometer semigroup, Fock representations, C*-algebras, Fredholm index, K-theory, essential extensions, Cuntz-Toeplitz algebra}

\begin{document}

\title{Finite Blaschke Symbols and the $K$-Theory of Scalar
Odometer $C^*$-Algebras}

\author{Mansi Anil Suryawanshi}
\address{Department of Mathematics, Technion---Israel Institute of Technology, Haifa 32000, Israel.}
\email{suryawanshi@campus.technion.ac.il; mansisuryawanshi1@gmail.com}

\date{\today}
\begin{abstract}
Let $O_n$ be the odometer semigroup, and let
$\cla_\xi=C^*(S_1,\ldots,S_n,W_\xi)\subseteq\clb(\clf_n^2)$
be the $C^*$-algebra generated by the left creation operators and the
scalar odometer map associated with a symbol $\xi\in\clf_n^2$.
We show that $\cla_\xi$ contains the compact operators for every scalar
symbol.
For an isometric scalar symbol, we prove that $W_\xi$ is Fredholm if
and only if the associated inner function is a finite Blaschke product. We further show that the image of $\cla_\xi$
in the Calkin algebra is canonically isomorphic to the odometer boundary
quotient $\clq(O_n)$.

If the associated finite Blaschke product has degree $d$, then $\operatorname{ind}(W_\xi)=-d.$
For $d\geq1$, we obtain
$K_0(\cla_\xi)\cong\mathbb Z\oplus\mathbb Z_{d(n-1)}$
and
$K_1(\cla_\xi)=0$,
whereas for $d=0$,
$K_0(\cla_\xi)\cong\mathbb Z^2$
and
$K_1(\cla_\xi)\cong\mathbb Z$.
Consequently, for fixed $n\geq2$, finite Blaschke symbols of distinct
degrees generate non-isomorphic $C^*$-algebras.
\end{abstract}

\maketitle

\section{Introduction}
\label{sec:introduction}
Let $n\geq 2$. The odometer semigroup $O_n$ is the unital semigroup
generated by $w,v_1,\ldots,v_n$, subject to the relations
\[
wv_i=v_{i+1},
\qquad 1\leq i<n,
\qquad\text{and}\qquad
wv_n=v_1w.
\]
These relations encode the carrying rule of the classical adding machine.
The semigroup is closely related to the positive Baumslag--Solitar monoid \cite{Foreman, mansi2025}
and admits a Zappa--Sz'ep product description; see 
\cite{Geb}. The corresponding classical topological odometer has
long played an important role in ergodic theory and measure-preserving
dynamics, with further connections to Toeplitz algebras, von Neumann
algebras, and semigroup $C^*$-algebras; see
\cite{davidson1996,douglas1998banach,li2019}.
From the operator-algebraic viewpoint, representations of $O_n$ belong to
the broader framework of semigroup and correspondence $C^*$-algebras,
index theory, boundary quotients, topological subshifts, and
scale-invariant dynamics; see
\cite{brownlowe2012,cuntz1977,julien2016,laca1996,li2022,nica1992,pimsner1997,pimsner1980}.
Its boundary quotient is canonically isomorphic to Cuntz's $n$-adic ring
$C^*$-algebra $\mathcal Q_n$; see \cite{li2019}.

Fock-space representations of $O_n$ were studied in
\cite{mansi2025}. Given a Hilbert space $\cle$ and a bounded operator
$
L:\cle\longrightarrow\clf_n^2\otimes\cle,
$
an associated odometer map
$
W_L\in\clb(\clf_n^2\otimes\cle)
$
was constructed. Together with the amplified left creation operators, these
maps provide a parametrization of the Fock representations of $O_n$.
Criteria for $W_L$ to be isometric, unitary, or Nica-covariant were also
obtained. The operator-theoretic structure of $W_L$, including its adjoint,
canonical decompositions, and Toeplitz realization, was further developed in
\cite{Suryawanshi2026}.

The present paper considers the scalar case and studies the $C^*$-algebra
generated by an odometer map together with the left creation operators.
For a scalar symbol $\xi\in\clf_n^2$, set
$$
\cla_\xi
=
C^*(S_1,\ldots,S_n,W_\xi)
\subseteq
\clb(\clf_n^2).
$$
Rather than studying only the individual operator $W_\xi$, we investigate
the compact ideal of $\cla_\xi$, its image in the Calkin algebra, its ideal
structure, and its $K$-theory.

Although the $K$-theory of the odometer boundary quotient
$\mathcal Q(O_n)$ is known, to the best of our knowledge the
concrete Toeplitz-type $C^*$-algebras generated by scalar
Fock-space representations of the odometer semigroup have not
previously been investigated from a $K$-theoretic viewpoint.
In particular, we show that the connecting map of the associated
compact extension records the finite Blaschke degree of the symbol.

Our first observation is independent of the symbol. The Cuntz--Toeplitz
algebra generated by $S_1,\ldots,S_n$ contains the rank-one matrix units
associated with the canonical Fock basis. Consequently, for every scalar symbol $\xi$,
$$
\clk(\clf_n^2)\subseteq\cla_\xi.
$$

We next determine which scalar isometric odometer maps are Fredholm.
Scalar isometric symbols are naturally identified with inner functions on
the unit disk. More precisely, if $\clm^\perp$ denotes the closed span of
$\{e_1^{\otimes p}:p\geq 0\}$, then the unitary
$$
U_1:\clm^\perp\longrightarrow H^2(\mathbb D),
\qquad
U_1(e_1^{\otimes p})=z^p,
$$
associates to an isometric scalar symbol $\xi$ an inner function
$
\Theta_\xi=U_1\xi.
$
The kernel of the adjoint of $W_\xi$ is identified with the classical
model space
$$
U_1(\ker W_\xi^*)
=
H^2(\mathbb D)\ominus\Theta_\xi H^2(\mathbb D).
$$
It follows that $W_\xi$ is Fredholm if and only if $\Theta_\xi$ is a
finite Blaschke product. If its degree is $d\geq 0$, then
$
\operatorname{ind}(W_\xi)=-d.
$
Equivalently,
$
I-W_\xi W_\xi^*
$
is a finite-rank projection of rank $d$. We refer to such a symbol as a
finite Blaschke symbol of degree $d$. The case $d=0$ corresponds to a
unimodular constant, in which case $W_\xi$ is unitary.

The compact-defect condition permits us to identify the Calkin image of
$\cla_\xi$. We obtain
$$
\cla_\xi/\clk(\clf_n^2)
\cong
\clq(O_n).
$$
In particular, every finite Blaschke symbol gives an essential extension
$$
0
\longrightarrow
\clk(\clf_n^2)
\longrightarrow
\cla_\xi
\overset{\Phi_\xi}{\longrightarrow}
\clq(O_n)
\longrightarrow
0.
$$
Using this extension, we compute the $K$-theory of $\cla_\xi$. 
For $d\geq 1$, we show
$$
K_1(\cla_\xi)=0,
\qquad
K_0(\cla_\xi)
\cong
\mathbb Z\oplus\mathbb Z_{d(n-1)}.
$$
The degree-zero case has different behavior:
$$
K_0(\cla_{\lambda\Omega})
\cong
\mathbb Z^2,
\qquad
K_1(\cla_{\lambda\Omega})
\cong
\mathbb Z,
\qquad
\lambda\in\mathbb T.
$$
Thus, for a finite Blaschke symbol $\xi$ of degree $d$, the $K$-groups are summarized as follows.
$$
\begin{array}{c|cc}
& d=0 & d\geq 1 \\ \hline
K_0(\cla_\xi)
& \mathbb Z^2
& \mathbb Z\oplus\mathbb Z_{d(n-1)}
\\[2mm]
K_1(\cla_\xi)
& \mathbb Z
& 0
\end{array}
$$

As a consequence, finite Blaschke symbols of different degrees give rise
to non-isomorphic $C^*$-algebras. These results provide
a degree-detection theorem, but not a complete classification. Symbols of
the same degree yield the same $K$-groups and the same distinguished unit
class. A natural next step toward classification is to compare the Busby invariants of the corresponding essential extensions.

The paper is organized as follows. Section~2 contains the necessary preliminaries. In Section~3, we prove that
$\cla_\xi$ contains the compact operators for every scalar symbol and
deduce irreducibility and essentiality. Section~4 characterizes the
Fredholm scalar odometer maps in terms of finite Blaschke products and
computes their indices. In Section~5, we identify the boundary quotient and
derive the nuclearity and ideal structure of $\cla_\xi$. Section~6
computes the $K$-groups, treats the degree-zero case, and establishes the
degree-detection and non-isomorphism results. 

\section{Preliminaries and notation}
\label{sec:preliminaries}

Throughout the paper, let \(n\geq 2\). We write \(\mathbb Z_m\) for the
cyclic group \(\mathbb Z/m\mathbb Z\), with the convention that
\(\mathbb Z_1=\{0\}\).

\subsection{The full Fock space and scalar odometer maps}

Let
\[
\clf_n^2
=
\mathbb C\Omega
\oplus
\bigoplus_{m\geq 1}(\mathbb C^n)^{\otimes m}
\]
be the full Fock space over \(\mathbb C^n\), where \(\Omega\) denotes the
vacuum vector. Let \(\mathbb F_n^+\) denote the unital free semigroup on
the generators \(1,\ldots,n\), with identity element
\(\varnothing\). For a word $\mu=\mu_1\mu_2\cdots\mu_m\in\mathbb F_n^+,$
we write $|\mu|=m,
\,
e_\mu
=
e_{\mu_1}\otimes e_{\mu_2}\otimes\cdots\otimes e_{\mu_m},$
and set \(e_\varnothing=\Omega\). Then
$\{e_\mu:\mu\in\mathbb F_n^+\}$
is the canonical orthonormal basis of \(\clf_n^2\).
For \(1\leq i\leq n\), the left creation operator
\(S_i\in\clb(\clf_n^2)\) is defined by
\[
S_ie_\mu=e_{i\mu},
\qquad
\mu\in\mathbb F_n^+.
\]
The operators \(S_1,\ldots,S_n\) are isometries with pairwise orthogonal
ranges. For a word
\(\mu=\mu_1\cdots\mu_m\in\mathbb F_n^+\), we use the notation
$S_\mu=S_{\mu_1}\cdots S_{\mu_m},
\,
S_\varnothing=I.$
The orthogonal projection onto the vacuum space \(\mathbb C\Omega\) is
\[
P_\Omega
=
I-\sum_{i=1}^nS_iS_i^*.
\]

The odometer semigroup \(O_n\) is the unital semigroup generated by
\(w,v_1,\ldots,v_n\), subject to the relations
\[
wv_i=v_{i+1},
(1\leq i<n) \qquad \text{ and }\qquad
wv_n=v_1w.
\]
The following class of operators provides the canonical Fock-space
representations of these relations.

\begin{definition}
\label{def:scalar-odometer-map}
Let \(\xi\in\clf_n^2\). The \emph{scalar odometer map with symbol
\(\xi\)} is the bounded operator \(W_\xi\in\clb(\clf_n^2)\) determined by
$W_\xi\Omega=\xi$
and the following carrying rule:
Let $\mu=\mu_1\mu_2\cdots\mu_m\in\mathbb F_n^+$
be a nonempty word. If \(\mu\neq n^m\), let
$r=\min\{j\in\{1,\ldots,m\}:\mu_j\neq n\}.$
Then
\[
W_\xi e_\mu
=
e_1^{\otimes(r-1)}
\otimes e_{\mu_r+1}
\otimes e_{\mu_{r+1}}
\otimes\cdots\otimes e_{\mu_m}.
\]
If \(\mu=n^m\), then
\[
W_\xi e_{n^m}
=
e_1^{\otimes m}\otimes\xi.
\]
\end{definition}
Equivalently, if $L_\xi:\mathbb C\longrightarrow\clf_n^2,
\,
L_\xi(\lambda)=\lambda\xi,$
then \(W_\xi\) is the odometer map \(W_{L_\xi}\) associated with the
operator-valued symbol \(L_\xi\). By the boundedness result for odometer
maps in \cite{Suryawanshi2026}, the above prescription defines an element
of \(\clb(\clf_n^2)\), and
\[
\|\xi\|
\leq
\|W_\xi\|
\leq
1+\|\xi\|.
\]
The defining formula immediately gives the odometer covariance relations
\[
W_\xi S_i=S_{i+1},
\, (1\leq i<n) \qquad \text{and} \qquad
W_\xi S_n=S_1W_\xi.
\]
Thus \((W_\xi,S_1,\ldots,S_n)\) is an operator representation of the
defining relations of \(O_n\).
We shall be particularly interested in those symbols for which
\(W_\xi\) is an isometry.

\begin{definition}
A vector \(\xi\in\clf_n^2\) is called an
\emph{isometric scalar symbol} if \(W_\xi\) is an isometry.
\end{definition}
We next introduce the one-variable subspaces that occur in the
characterization of isometric scalar symbols. Set
\[
\clm
=
\overline{\operatorname{span}}
\bigl\{
e_\mu:
|\mu|\geq1
\text{ and some letter of }\mu\text{ is different from }1
\bigr\}.
\]
Equivalently, $\clm^\perp
=
\overline{\operatorname{span}}
\{e_1^{\otimes p}:p\geq0\}.$
Define a unitary operator $U_1:\clm^\perp\longrightarrow H^2(\D)$
by $U_1(e_1^{\otimes p})=z^p,
\, p\geq0.$
Then
\[
U_1\bigl(S_1|_{\clm^\perp}\bigr)U_1^*=M_z,
\]
where \(M_z\) denotes the unilateral shift on \(H^2(\D)\).
For \(\xi\in\clm^\perp\), define
$\Theta_\xi=U_1\xi\in H^2(\D).$
The following is the scalar specialization of
\cite[Theorem~4.4]{Suryawanshi2026}.

\begin{theorem}
\label{thm:operator_characterization_symbol}
Let \(\xi\in\clf_n^2\). Then the following statements hold.
\begin{enumerate}
    \item The operator \(W_\xi\) is an isometry if and only if $\xi\in\clm^\perp$
    and \(\Theta_\xi=U_1\xi\) is an inner function.

    \item The operator \(W_\xi\) is unitary if and only if $ \xi=\lambda\Omega$
    for some \(\lambda\in\mathbb T\).
\end{enumerate}
\end{theorem}
Thus the scalar isometric symbols are naturally identified with scalar
inner functions on the unit disk. The kernel of the adjoint has a
particularly useful description.

\begin{proposition}
\label{prop:kernel_adjoint}
Let \(\xi\in\clm^\perp\), and suppose that
\(\Theta_\xi=U_1\xi\) is inner. Then
\[
\ker W_\xi^*
=
\left\{
f\in\clm^\perp:
\left\langle S_1^{*p}f,\xi\right\rangle=0
\text{ for every }p\geq0
\right\}
=
\clm^\perp
\ominus
\overline{\operatorname{span}}
\{S_1^p\xi:p\geq0\}.
\]
Applying \(U_1\) gives
\[
U_1(\ker W_\xi^*)
=
H^2(\D)\ominus\Theta_\xi H^2(\D).
\]
\end{proposition}

\begin{proof}
The first equality is the scalar specialization of
\cite[Proposition~7.1]{Suryawanshi2026}. Since
$\left\langle S_1^{*p}f,\xi\right\rangle
=
\left\langle f,S_1^p\xi\right\rangle,$
the second equality follows. Moreover,
\[
U_1(S_1^p\xi)=z^p\Theta_\xi,
\qquad p\geq0.
\]
Since \(\Theta_\xi\) is inner,
$\overline{\operatorname{span}}
\{z^p\Theta_\xi:p\geq0\}
=
\Theta_\xi H^2(\D),$
which proves the final assertion.
\end{proof}

\subsection{The generated \(C^*\)-algebra and the Calkin algebra}

For a scalar symbol \(\xi\in\clf_n^2\), define
\[
\cla_\xi
=
C^*(S_1,\ldots,S_n,W_\xi)
\subseteq
\clb(\clf_n^2).
\]
This is the principal \(C^*\)-algebra studied in the paper.
For vectors \(x,y\) in a Hilbert space \(\clh\), let
\(\theta_{x,y}\in\clb(\clh)\) denote the rank-one operator
\[
\theta_{x,y}(h)=\langle h,y\rangle x,
\qquad h\in\clh.
\]
We denote by \(\clk(\clh)\) the \(C^*\)-algebra of compact operators on
\(\clh\).
For a Hilbert space \(\clh\), let us denote the Calkin algebra by
\[
\clq(\clh)
=
\clb(\clh)/\clk(\clh),
\]
and let $\pi_{\clh}:\clb(\clh)\longrightarrow\clq(\clh)$
be the quotient map
\[
\pi_{\clh}(T)=T+\clk(\clh).
\]
When \(\clh=\clf_n^2\), we write simply
$\pi=\pi_{\clf_n^2}.$
Recall that an operator \(T\in\clb(\clh)\) is Fredholm if and only if
\(\pi_{\clh}(T)\) is invertible in \(\clq(\clh)\) (see \cite{atkinson1951}). Its Fredholm index is
\[
\operatorname{ind}(T)
=
\dim\ker T-\dim\ker T^*.
\]
An operator \(T\in\clb(\clh)\) is called
\emph{essentially unitary} if \(\pi_{\clh}(T)\) is unitary, equivalently,
if
$I-T^*T\in\clk(\clh)$ and $
I-TT^*\in\clk(\clh).$
In particular, if \(T\) is an isometry, then \(T\) is essentially unitary
if and only if
$I-TT^*\in\clk(\clh).$

\subsection{$K$-theoretic conventions}
\label{subsec:k-theory-conventions}

For a projection $p$ and a unitary $v$, we denote their classes in
$K$-theory by $[p]_0$ and $[v]_1$, respectively. We follow the standard
conventions for $K$-theory and extensions of $C^*$-algebras; see
\cite{blackadar1998,bdf1977}.
A short exact sequence
$$
0
\longrightarrow
\mathcal I
\overset{i}{\longrightarrow}
\mathcal A
\overset{q}{\longrightarrow}
\mathcal B
\longrightarrow
0
$$
induces the six-term exact sequence
$$
\begin{array}{ccccc}
K_0(\mathcal I)
&
\overset{i_*}{\longrightarrow}
&
K_0(\mathcal A)
&
\overset{q_*}{\longrightarrow}
&
K_0(\mathcal B)
\\[2mm]
\uparrow\,\partial_1
&&&&
\downarrow\,\partial_0
\\[2mm]
K_1(\mathcal B)
&
\overset{q_*}{\longleftarrow}
&
K_1(\mathcal A)
&
\overset{i_*}{\longleftarrow}
&
K_1(\mathcal I).
\end{array}
$$
For an infinite-dimensional separable Hilbert space $\clh$,
$$
K_0(\clk(\clh))\cong\mathbb Z,
\qquad
K_1(\clk(\clh))=0,
$$
where the class of a rank-one projection corresponds to
$1\in\mathbb Z$.
We use the convention that, for an extension by $\clk(\clh)$, if
$v$ is a unitary in the quotient and $T$ is a Fredholm lift of $v$,
then
$$
\partial_1([v]_1)
=
\operatorname{ind}(T)
=
\dim\ker T-\dim\ker T^*.
$$

\section{The compact ideal}
Retain the notation
$\cla_\xi=C^*(S_1,\ldots,S_n,W_\xi).$
We begin by showing that \(\cla_\xi\) always contains the compact
operators, independently of the choice of \(\xi\).

\begin{proposition}\label{prop:compact-containment}
For every scalar symbol $\xi\in\clf_n^2$, $\clk(\clf_n^2)\subseteq\cla_\xi$.
\end{proposition}

\begin{proof}
Let $P_\Omega=I-\sum_{i=1}^n S_iS_i^*\in C^*(S_1,\ldots,S_n)$ be the
rank-one projection onto the vacuum vector. With respect to the canonical
Fock basis $\{e_\mu:\mu\in\mathbb F_n^+\}$, the operators
$S_\mu P_\Omega S_\nu^*$, $\mu,\nu\in\mathbb F_n^+$, are exactly the rank-one
matrix units $\theta_{e_\mu,e_\nu}$, so their linear span is dense in
$\clk(\clf_n^2)$. Since each $S_\mu P_\Omega S_\nu^*\in
C^*(S_1,\ldots,S_n)\subseteq\cla_\xi$, taking norm closure gives
$\clk(\clf_n^2)\subseteq\cla_\xi$.
\end{proof}

\begin{corollary}
\label{cor:irreducible-essential}
For every scalar symbol \(\xi\in\clf_n^2\), the algebra
\(\cla_\xi\) acts irreducibly on \(\clf_n^2\), and
\(\clk(\clf_n^2)\) is an essential ideal of \(\cla_\xi\).
\end{corollary}
\begin{proof} By Proposition~\ref{prop:compact-containment}, $ \clk(\clf_n^2)\subseteq \cla_\xi. $ Since the compact operators act irreducibly on \(\clf_n^2\), every reducing subspace for \(\cla_\xi\) is also reducing for \(\clk(\clf_n^2)\). It must therefore be either \(\{0\}\) or \(\clf_n^2\). Hence \(\cla_\xi\) acts irreducibly on \(\clf_n^2\). To prove essentiality, let \(J\subseteq\cla_\xi\) be a nonzero closed two-sided ideal, and choose \(0\neq A\in J\). There exists \(h\in\clf_n^2\) such that \(Ah\neq0\). Fix a nonzero vector \(k\in\clf_n^2\), and let \(\theta_{h,k}\) denote the corresponding rank-one operator. Since \[ \theta_{h,k}\in\clk(\clf_n^2)\subseteq\cla_\xi, \] we have \[ 0\neq A\theta_{h,k} = \theta_{Ah,k} \in J\cap\clk(\clf_n^2). \] Thus every nonzero closed two-sided ideal of \(\cla_\xi\) has nonzero intersection with \(\clk(\clf_n^2)\). Therefore \(\clk(\clf_n^2)\) is an essential ideal of \(\cla_\xi\). \end{proof}

\section{Finite Blaschke symbols and the Fredholm index}
We characterize the Fredholm scalar odometer maps and compute their
Fredholm indices. Recall that if \(\xi\) is an isometric scalar symbol,
then
$\Theta_\xi=U_1\xi$
is inner.

\begin{definition}
\label{def:finite-blaschke-symbol}
An isometric scalar symbol \(\xi\) is called a
\emph{finite Blaschke symbol of degree \(d\geq0\)} if
\(\Theta_\xi=U_1\xi\) is a finite Blaschke product of degree \(d\); that is,
\[
\Theta_\xi(z)
=
\lambda
\prod_{j=1}^d
\frac{z-a_j}{1-\overline{a_j}z},
\qquad
\lambda\in\mathbb T,
\quad
a_1,\ldots,a_d\in\mathbb D,
\]
where the zeros are counted with multiplicity. When \(d=0\), the product
is interpreted as the empty product, so that \(\Theta_\xi\) is a
unimodular constant.
We write \(\deg(\xi)=d\).
\end{definition}

\begin{example}
For $d \geq 0,$ with the convention \(e_1^{\otimes0}=\Omega\), the vector $\xi = e_1^{\otimes d}$ has finite Blaschke degree $d$, since $U_1\xi = z^d$. More generally, if
$\Theta(z) = \lambda \prod_{r=1}^d \frac{z - a_r}{1 - \overline{a_r} z}, \, |\lambda| = 1,\ a_1,\dots,a_d \in \mathbb D,$
is any finite Blaschke product of degree $d$, with Taylor coefficients $\Theta(z) = \sum_{p \ge 0} c_p z^p$, then
\[
\xi := U_1^{-1}\Theta = \sum_{p \ge 0} c_p\, e_1^{\otimes p} \in \mathcal M^\perp
\]
is a finite Blaschke symbol of degree \(d\).
\end{example}
\begin{theorem}\label{thm:index}
Let $\xi$ be an isometric scalar symbol.
Then the following statements are equivalent:
\begin{enumerate}
    \item $W_\xi$ is Fredholm;

    \item $I-W_\xi W_\xi^*$ is compact;

    \item \(\xi\) is a finite Blaschke symbol.
\end{enumerate}
If these conditions hold and $\deg(\xi)=d$, then
$\operatorname{ind}(W_\xi)=-d.$
\end{theorem}

\begin{proof}
Let $\Theta:=U_1\xi.$
Since $\xi$ is an isometric scalar symbol, Theorem~\ref{thm:operator_characterization_symbol}
implies that $\Theta$ is an inner function.
By Proposition \ref{prop:kernel_adjoint}, 
$\ker W_\xi^*
=
\mathcal M^\perp
\ominus
\overline{\operatorname{span}}\{S_1^p\xi:p\geq 0\}.$
Applying $U_1$ and using $U_1S_1^p\xi=z^p\Theta,$
we obtain
\[
U_1(\ker W_\xi^*)
=
H^2(\mathbb D)\ominus \Theta H^2(\mathbb D), \qquad \text{and} \qquad \dim\ker W_\xi^*
=
\dim\bigl(H^2(\mathbb D)\ominus \Theta H^2(\mathbb D)\bigr).
\]
Since $W_\xi$ is an isometry, $\ker W_\xi=\{0\}$ and $\operatorname{ran}W_\xi$
is closed. Hence $W_\xi$ is Fredholm if and only if $\ker W_\xi^*$ is
finite-dimensional, which is equivalent to the model space
$H^2(\mathbb D)\ominus \Theta H^2(\mathbb D)$
being finite-dimensional.
Since $\Theta$ is inner, the standard model-space theorem implies that this
is equivalent to $\Theta$ being a finite Blaschke product. Moreover, in that case,
\[
\dim\bigl(H^2(\mathbb D)\ominus \Theta H^2(\mathbb D)\bigr)
=
\deg\Theta.
\]
Thus $W_\xi$ is Fredholm if and only if $\xi$ has finite Blaschke degree.
When this holds, $\dim\ker W_\xi^*=\deg(\xi).$ 
Since $\ker W_\xi=\{0\}$, we get
\[
\operatorname{ind}(W_\xi)
=
\dim\ker W_\xi-\dim\ker W_\xi^*
=
0-\deg(\xi)
=
-\deg(\xi).
\]
Moreover, since $W_\xi$ is an isometry,
$
I-W_\xi W_\xi^*
=
P_{(\operatorname{ran}W_\xi)^\perp}
=
P_{\ker W_\xi^*}.
$
An orthogonal projection is compact if and only if it has finite
rank. 
Therefore,
\[
I-W_\xi W_\xi^*\in\clk(\clf_n^2)
\quad\Longleftrightarrow\quad
\deg(\xi)=\dim\ker W_\xi^*<\infty.
\]
\end{proof}

\section{The boundary quotient and structural consequences}
\label{sec:boundary-quotient}
By Theorem~\ref{thm:index}, finite Blaschke symbols are precisely the
isometric scalar symbols for which the defect projection $I-W_\xi W_\xi^*$
is compact. We first identify the quotient of \(\cla_\xi\) under this
compact-defect hypothesis and then record some structural consequences
for finite Blaschke symbols.

Let $\pi:\clb(\clf_n^2)\longrightarrow
\clb(\clf_n^2)/\clk(\clf_n^2)$
denote the Calkin quotient map, and set
$s_i=\pi(S_i)$, $w=\pi(W_\xi)$. We denote by
$\clq(O_n)$ the boundary quotient algebra of the odometer semigroup, that is,
the universal $C^*$-algebra generated by a unitary $u$ and isometries
$t_1,\ldots,t_n$ satisfying
\[
t_i^*t_j=\delta_{ij}1,\qquad
\sum_{i=1}^n t_it_i^*=1,
\]
and the odometer covariance relations
\[
ut_i=t_{i+1},\qquad 1\leq i<n,
\qquad\text{and}\qquad
ut_n=t_1u.
\]
It is known that \(\clq(O_n)\) is canonically isomorphic to Cuntz's
\(n\)-adic ring \(C^*\)-algebra \(\mathcal Q_n\); in particular, it is
simple, nuclear, and purely infinite. See \cite{li2019}.

\begin{theorem}\label{thm:quotient-isomorphism}
Assume that $W_\xi$ is an isometry and that $I-W_\xi W_\xi^*$ is compact.
Then there exists a unique $*$-isomorphism
\[
\rho_\xi:\clq(O_n)\overset{\cong}{\longrightarrow}\pi(\cla_\xi)\cong\cla_\xi/\clk(\clf_n^2),
\qquad
\rho_\xi(u)=w,\quad \rho_\xi(t_i)=s_i\ (1\le i\le n).
\]
\end{theorem}

\begin{proof}
    Applying $\pi$ to $S_i^*S_j=\delta_{ij}I$ gives $s_i^*s_j=\delta_{ij}1$.
Since $I-\sum_{i=1}^n S_iS_i^*=P_\Omega$ is compact, $\pi(P_\Omega)=0$, so
$\sum_{i=1}^n s_is_i^*=1.$ As $W_\xi$ is an isometry, $w^*w=1$; since
$I-W_\xi W_\xi^*$ is compact by hypothesis, $ww^*=1$ as well, so $w$ is
unitary. Finally, applying $\pi$ to the odometer covariance relations
$W_\xi S_i=S_{i+1}$ ($1\le i<n$) and $W_\xi S_n=S_1W_\xi$ gives
$ws_i=s_{i+1}$ and $ws_n=s_1w$.

By the universal property of $\clq(O_n)$, the assignment $u\mapsto w$, $t_i\mapsto s_i$ extends to a unital $*$-homomorphism $\rho_\xi:\clq(O_n)\to\pi(\cla_\xi)$. Its range contains $w,s_1,\ldots,s_n$, and since
\[
\pi(\cla_\xi)=C^*(\pi(S_1),\ldots,\pi(S_n),\pi(W_\xi))=C^*(s_1,\ldots,s_n,w),
\]
$\rho_\xi$ is surjective. By Proposition~\ref{prop:compact-containment}, $\ker(\pi|_{\cla_\xi})=\clk(\clf_n^2)$, so $\pi(\cla_\xi)\cong\cla_\xi/\clk(\clf_n^2)$.
Since \(\clq(O_n)\) is simple and \(\rho_\xi\) is a nonzero unital
homomorphism, \(\rho_\xi\) is injective.
Thus $\rho_\xi$ is a $*$-isomorphism.
\end{proof}

Combining Theorems~\ref{thm:index} and
\ref{thm:quotient-isomorphism}, we obtain an essential extension of
\(\clq(O_n)\) by the compact operators for every finite Blaschke symbol.

\begin{corollary}\label{cor:finite-degree-extension}
Let $\xi$ be a finite Blaschke symbol of degree $d\ge 0$. Then there is a short exact sequence
\[
0\longrightarrow \clk(\clf_n^2)
\longrightarrow \cla_\xi
\overset{\Phi_\xi}{\longrightarrow}
\clq(O_n)
\longrightarrow 0,
\qquad
\Phi_\xi:=\rho_\xi^{-1}\circ\pi|_{\cla_\xi},
\]
exhibiting $\cla_\xi$ as an essential extension of $\clq(O_n)$ by the compact operators.
\end{corollary}

\begin{proof}
By Theorem~\ref{thm:index}, $I - W_\xi W_\xi^* \in \clk(\clf_n^2)$, so by Theorem~\ref{thm:quotient-isomorphism} there is a $*$-isomorphism $\rho_\xi:\clq(O_n)\to\pi(\cla_\xi)$. Hence
\[
\Phi_\xi:=\rho_\xi^{-1}\circ\pi|_{\cla_\xi}:\cla_\xi\longrightarrow\clq(O_n)
\]
is a surjective $*$-homomorphism. Moreover,
\[
\ker\Phi_\xi
=
\ker(\pi|_{\cla_\xi})
=
\cla_\xi\cap\clk(\clf_n^2)
=
\clk(\clf_n^2),
\]
where the final equality follows from
Proposition~\ref{prop:compact-containment}. Thus the displayed sequence
is exact. Finally, Corollary~\ref{cor:irreducible-essential} shows that
\(\clk(\clf_n^2)\) is an essential ideal of \(\cla_\xi\).
\end{proof}

\begin{corollary}
For every finite Blaschke symbol $\xi$, the $C^*$-algebra $\cla_\xi$ is
nuclear, and hence exact.
\end{corollary}

\begin{proof}
Consider the short exact sequence from Corollary~\ref{cor:finite-degree-extension}.
The algebra $\clk(\clf_n^2)$ of compact operators is nuclear. Moreover,
$\clq(O_n)$ is canonically isomorphic to Cuntz's $n$-adic ring
$C^*$-algebra $\mathcal Q_n$, which is nuclear; see, for example,
\cite{li2019}.
Since nuclearity is preserved under extensions, \(\cla_\xi\) is nuclear.
In particular, \(\cla_\xi\) is exact.
\end{proof}

\begin{proposition}\label{prop:ideal-lattice}
Let \(\xi\) be a finite Blaschke symbol. Then the closed
two-sided ideals of $\cla_\xi$ are exactly
\[
\{0\},\qquad \clk(\clf_n^2),\qquad \cla_\xi.
\]
\end{proposition}
\begin{proof}
Let $J\subseteq\cla_\xi$ be a nonzero closed ideal. By
Corollary~\ref{cor:finite-degree-extension}, $\clk(\clf_n^2)$ is an
essential ideal of $\cla_\xi$, so $J\cap\clk(\clf_n^2)\neq\{0\}$. Since
$\clf_n^2$ is infinite-dimensional, $\clk(\clf_n^2)$ is
simple as a (non-unital) $C^*$-algebra, so its only closed ideals are
$\{0\}$ and $\clk(\clf_n^2)$ itself; hence
\[
J\cap\clk(\clf_n^2)=\clk(\clf_n^2),\qquad\text{i.e.}\qquad
\clk(\clf_n^2)\subseteq J.
\]
Thus $J/\clk(\clf_n^2)$ is a closed ideal of
$\cla_\xi/\clk(\clf_n^2)$. The homomorphism \(\Phi_\xi\) induces a \(*\)-isomorphism
$$
\widetilde{\Phi}_\xi:
\cla_\xi/\clk(\clf_n^2)
\overset{\cong}{\longrightarrow}
\clq(O_n).
$$
Under this isomorphism, $J/\clk(\clf_n^2)$
corresponds to a closed ideal of $\clq(O_n)$. 
Since $\clq(O_n)$ is simple, this ideal is either $\{0\}$ or all of
$\clq(O_n)$, giving $J=\clk(\clf_n^2)$ or $J=\cla_\xi$ respectively. Together with the zero ideal, these are all the closed two-sided ideals
of \(\cla_\xi\).
\end{proof}

\section{$K$-theory and extension invariants of $\cla_\xi$} \label{sec: K theory}
Let \(\xi\) be a finite Blaschke symbol. By
Corollary~\ref{cor:finite-degree-extension}, there is an essential
extension
\[
0\longrightarrow\clk(\clf_n^2)
\overset{i}{\longrightarrow}\cla_\xi
\overset{\Phi_\xi}{\longrightarrow}\clq(O_n)
\longrightarrow0.
\]
We use the identification
\[
\left(
K_0(\clq(O_n)),
[1_{\clq(O_n)}]_0,
K_1(\clq(O_n))
\right)
\cong
\left(
\Z\oplus\Z_{n-1},
(0,\overline{1}),
\Z
\right);
\]
see \cite{BOS2018}. The boundary-map convention is the one fixed in
Subsection~\ref{subsec:k-theory-conventions}.

Before computing the connecting map for a general finite Blaschke
symbol, we identify a natural generator of
$K_1(\clq(O_n))$. Let
$\xi_1=e_1$, so that $\Theta_{\xi_1}(z)=z$ and
$\deg(\xi_1)=1$. For the extension associated with
$\cla_{\xi_1}$, we have
$
\Phi_{\xi_1}(W_{\xi_1})=u,
$
where $u\in\clq(O_n)$ is the canonical odometer unitary. Thus
$W_{\xi_1}$ is a Fredholm lift of $u$, and our boundary-map
convention gives
$$
\partial_1^{\xi_1}([u]_1)
=
\operatorname{ind}(W_{\xi_1})
=
-1.
$$
To see that $[u]_1$ is a generator, let $g$ be a generator of
$K_1(\clq(O_n))\cong\mathbb Z$ and write
$[u]_1=mg$ for some $m\in\mathbb Z$. Then
$$
-1
=
\partial_1^{\xi_1}([u]_1)
=
m\,\partial_1^{\xi_1}(g).
$$
Hence $m=\pm1$, and therefore $[u]_1$ generates
$K_1(\clq(O_n))$. We henceforth identify
$K_1(\clq(O_n))$ with $\mathbb Z$ by declaring
$[u]_1$ to correspond to $1$.

\begin{theorem}\label{thm:k-theory}
Let \(n\geq2\), and let \(\xi\) be a finite Blaschke symbol of degree
\(d\geq1\). Then
\[
K_1(\cla_\xi)=0 \qquad \text{ and } \qquad K_0(\cla_\xi)
\cong
\Z\oplus\Z_{d(n-1)}.
\]
More precisely, $\operatorname{Tor}\bigl(K_0(\cla_\xi)\bigr)
=
\bigl\langle [1_{\cla_\xi}]_0\bigr\rangle
\cong
\Z_{d(n-1)},$
and the isomorphism $K_0(\cla_\xi)
\cong
\Z\oplus\Z_{d(n-1)}$
may be chosen so that
\[
[1_{\cla_\xi}]_0
\longmapsto
(0,\overline{1}).
\]
\end{theorem}

\begin{proof}
By Corollary~\ref{cor:finite-degree-extension}, there is a short exact sequence
\[
0
\longrightarrow
\clk(\clf_n^2)
\overset{i}{\longrightarrow}
\cla_\xi
\overset{\Phi_\xi}{\longrightarrow}
\clq(O_n)
\longrightarrow
0.
\]
The associated six-term exact sequence is
\[
\begin{tikzcd}
K_0(\clk(\clf_n^2))
    \arrow[r, "i_*"]
&
K_0(\cla_\xi)
    \arrow[r, "\Phi_{\xi *}"]
&
K_0(\clq(O_n))
    \arrow[d, "\partial_0"]
\\
K_1(\clq(O_n))
    \arrow[u, "\partial_1"]
&
K_1(\cla_\xi)
    \arrow[l, "\Phi_{\xi *}"']
&
K_1(\clk(\clf_n^2))
    \arrow[l, "i_*"']
\end{tikzcd}.
\]
Since $\clf_n^2$ is infinite-dimensional and separable,
\[
K_0(\clk(\clf_n^2))\cong\Z,
\qquad
K_1(\clk(\clf_n^2))=0,
\]
where the class $[P_\Omega]_0$ of the rank-one vacuum projection is
identified with $1\in\Z$.
We also use the known $K$-theory of the odometer boundary quotient:
\[
\left(
K_0(\clq(O_n)),
[1_{\clq(O_n)}]_0,
K_1(\clq(O_n))
\right)
\cong
\left(
\Z\oplus\Z_{n-1},
(0,\overline{1}),
\Z
\right).
\]
In particular, the class $[1_{\clq(O_n)}]_0$ has exact order $n-1$.
Let $u\in\clq(O_n)$ be the canonical odometer unitary. By the
preceding observation, $[u]_1$ generates
$K_1(\clq(O_n))\cong\mathbb Z$. Since
$
\Phi_\xi(W_\xi)=u,
$
the operator $W_\xi$ is a Fredholm lift of $u$. Therefore,
$$
\partial_1([u]_1)
=
\operatorname{ind}(W_\xi)
=
-d.
$$
Hence
$
\partial_1:\mathbb Z\longrightarrow\mathbb Z
$
is multiplication by $-d$.
Since $d\geq 1$, the homomorphism $\partial_1$ is injective. By
exactness at $K_1(\clq(O_n))$,
\[
\operatorname{im}
\bigl(
\Phi_{\xi *}:K_1(\cla_\xi)\to K_1(\clq(O_n))
\bigr)
=
\ker\partial_1
=
\{0\}.
\]
On the other hand, exactness at $K_1(\cla_\xi)$ and the equality
$K_1(\clk(\clf_n^2))=0$ show that $\Phi_{\xi *}$ is injective on
$K_1(\cla_\xi)$. Therefore
\[
K_1(\cla_\xi)=0.
\]
Exactness at $K_0(\clk(\clf_n^2))$ gives $\ker i_*
=
\operatorname{im}\partial_1
=
d\Z.$
Hence
\[
\operatorname{im}i_*
\cong
\Z/d\Z
=
\Z_d.
\]
Moreover, $y:=i_*([P_\Omega]_0)=i_*(1)$
generates $\operatorname{im}i_*$. Since
\[
my=0
\quad\Longleftrightarrow\quad
i_*(m)=0
\quad\Longleftrightarrow\quad
m\in d\Z,
\]
the smallest positive integer annihilating $y$ is $d$. Therefore $y$
has exact order $d$.

Because $K_1(\clk(\clf_n^2))=0$, the connecting map $\partial_0:
K_0(\clq(O_n))
\longrightarrow
K_1(\clk(\clf_n^2))$
is zero. Exactness therefore implies that $\Phi_{\xi *}:
K_0(\cla_\xi)
\longrightarrow
K_0(\clq(O_n))$
is surjective. We consequently obtain a short exact sequence
\[
0
\longrightarrow
\Z_d
\longrightarrow
K_0(\cla_\xi)
\overset{\Phi_{\xi *}}{\longrightarrow}
\Z\oplus\Z_{n-1}
\longrightarrow
0.
\]
We next separate the free summand. Define
\[
r
:=
\operatorname{pr}_{\Z}\circ\Phi_{\xi *}:
K_0(\cla_\xi)\longrightarrow\Z,
\]
where $\operatorname{pr}_{\Z}:
\Z\oplus\Z_{n-1}\longrightarrow\Z$
is projection onto the first summand. Since $\Phi_{\xi *}$ and
$\operatorname{pr}_{\Z}$ are surjective, so is $r$. Hence 
$$0\longrightarrow\ker r\longrightarrow K_0(\cla_\xi)\overset{r}{\longrightarrow}\Z\longrightarrow 0$$
is a short exact sequence. Since $\Z$ is free and therefore projective, this sequence splits. Thus $K_0(\cla_\xi)\cong\ker r\oplus\Z.$ Setting $G:=\ker r$, we obtain
\[
K_0(\cla_\xi)
\cong
\Z\oplus G.
\]
The restriction of $\Phi_{\xi *}$ to $G$ has range
$\{0\}\oplus\Z_{n-1}$. Indeed, if
$(0,\overline m)\in\{0\}\oplus\Z_{n-1}$, the surjectivity of
$\Phi_{\xi *}$ provides $a\in K_0(\cla_\xi)$ such that
\[
\Phi_{\xi *}(a)=(0,\overline m).
\]
Then $r(a)=0$, and hence $a\in G$. Moreover,
$\ker\bigl(\Phi_{\xi *}|_G\bigr)
=
\ker\Phi_{\xi *}
=
\operatorname{im}i_*
\cong
\Z_d.$
Thus $G$ fits into the short exact sequence
\[
0
\longrightarrow
\Z_d
\longrightarrow
G
\longrightarrow
\Z_{n-1}
\longrightarrow
0.
\]
In particular, $G$ is finite and $
|G|=d(n-1).$
It remains to determine the structure of $G$. Set
\[
x:=[1_{\cla_\xi}]_0
\qquad\text{and}\qquad
y:=i_*\bigl([P_\Omega]_0\bigr).
\]
Since $P_\Omega
=
I-\sum_{j=1}^nS_jS_j^*,$
we have $y
=
x-\sum_{j=1}^n[S_jS_j^*]_0$ in $K_0(\cla_\xi)$.
For each $1\leq j\leq n$, the operator $S_j$ is an isometry, and
therefore
$S_jS_j^*$ is Murray--von Neumann equivalent to $I$, with
implementing partial isometry $S_j$. Consequently,
\[
[S_jS_j^*]_0
=
[S_j^*S_j]_0
=
[1_{\cla_\xi}]_0
=
x.
\]
It follows that $y
=
x-nx
=
-(n-1)x.$
Therefore
\begin{equation}\label{eq:vacuum-unit-relation}
i_*([P_\Omega]_0)
=
-(n-1)[1_{\cla_\xi}]_0.
\end{equation}
Since $y$ has exact
order $d$,
we obtain $d(n-1)x
=
-dy
=
0.$
Thus the order of $x$ divides $d(n-1)$.
Conversely, suppose that $kx=0$
for some $k\in\Z$. Applying $\Phi_{\xi *}$ gives
\[
0
=
\Phi_{\xi *}(kx)
=
k[1_{\clq(O_n)}]_0.
\]
The class $[1_{\clq(O_n)}]_0$ has exact order $n-1$, and hence $n-1\mid k.$
Write $k=(n-1)\ell$
for some $\ell\in\Z$. Then
\[
0
=
kx
=
\ell(n-1)x
=
-\ell y.
\]
Since $y$ has exact order $d$, it follows that
$d\mid\ell.$
Therefore $d(n-1)\mid k.$
This proves that $x=[1_{\cla_\xi}]_0$ has exact order $d(n-1)$.

Finally $\Phi_{\xi *}(x)
=
[1_{\clq(O_n)}]_0
=
(0,\overline{1}),$
so $r(x)
=
\operatorname{pr}_{\Z}\bigl(\Phi_{\xi *}(x)\bigr)
=
0.$
Hence $x\in G$. The cyclic subgroup $\langle x\rangle$ of $G$ has
order $d(n-1)$, while $G$ itself has order $d(n-1)$. It follows that
\[
G
=
\langle x\rangle
\cong
\Z_{d(n-1)}.
\]
Consequently,
\[
K_0(\cla_\xi)
\cong
\Z\oplus\Z_{d(n-1)} \quad \text{and} \quad
\operatorname{Tor}\bigl(K_0(\cla_\xi)\bigr)
=
\langle[1_{\cla_\xi}]_0\rangle
\cong
\Z_{d(n-1)}.
\]
Choosing a generator of the free summand arising from a splitting of
$r$, the resulting isomorphism may be arranged so that
\[
[1_{\cla_\xi}]_0
\longmapsto
(0,\overline{1}).
\]
\end{proof}
\begin{corollary}\label{cor:n-2}
Let \(n=2\), and let \(\xi\) be a finite Blaschke symbol of degree
\(d\geq1\). Then
\[
K_1(\cla_\xi)=0,
\qquad
K_0(\cla_\xi)\cong\Z\oplus\Z_d.
\]
\end{corollary}

The assumption $d\geq 1$ in Theorem~\ref{thm:k-theory} is essential, as the
degree-zero case has different $K$-theoretic behavior.

\begin{proposition}\label{prop:degree-zero}
Let \(\xi=\lambda\Omega\), where \(\lambda\in\mathbb T\). Then
$K_0(\cla_\xi)\cong\Z^2,
\,
K_1(\cla_\xi)\cong\Z.$ Moreover, the isomorphism on $K_0$ may be chosen so that
\[
[1_{\mathcal A_\xi}]_0\longmapsto(1,0).
\]
\end{proposition}
\begin{proof}
Suppose that \(\xi=\lambda\Omega\), where \(\lambda\in\mathbb T\).
Equivalently, \(\Theta_\xi\) is the constant inner function \(\lambda\),
viewed as a finite Blaschke product of degree zero.
By
Theorem~\ref{thm:operator_characterization_symbol}, $W_\xi$ is unitary. Hence $\operatorname{ind}(W_\xi)=0,$
and the connecting map $\partial_1:K_1(\clq(O_n))\longrightarrow K_0(\clk(\clf_n^2))$
is the zero map.
The six-term exact sequence therefore gives
$K_1(\cla_\xi)\cong \Z,$
in sharp contrast to the $d\ge1$ case, where $K_1(\cla_\xi)$ always vanishes.
We now compute $K_0(\cla_\xi)$. Since the boundary map is zero, the six-term exact sequence gives
\[
0
\longrightarrow
\mathbb Z
\xrightarrow{i_*}
K_0(\cla_\xi)
\longrightarrow
\mathbb Z\oplus\mathbb Z_{n-1}
\longrightarrow
0.
\]
Again, $i_*(1)=-(n-1)[1_{\cla_\xi}]_0.$
Since $i_*$ is injective, the class $[1_{\cla_\xi}]_0$ has infinite order.
Set
$
x=[1_{\cla_\xi}]_0
$
and let
$
H
=\Phi_{\xi *}^{-1}
\bigl({0}\oplus\Z_{n-1}\bigr).
$
Since
$
\Phi_{\xi *}(x)
=[1_{\clq(O_n)}]_0
=(0,\overline{1}),
$
the element $x$ maps to a generator of the torsion summand of
$K_0(\clq(O_n))$.
Let $a\in H$. Then
$
\Phi_{\xi *}(a)
=m\Phi_{\xi *}(x)
$
for some $m\in\Z$. Hence
$
a-mx
\in
\ker\Phi_{\xi *}
=\operatorname{im}i_*.
$
Since
$
\operatorname{im}i_*
=\bigl\langle -(n-1)x\bigr\rangle
\subseteq
\langle x\rangle,
$
we obtain $a\in\langle x\rangle$. Conversely, since
$\Phi_{\xi *}(x)\in \{0\}\oplus\Z_{n-1}$, we have
$\langle x\rangle\subseteq H$. Thus
$
H=\langle x\rangle\cong\Z.
$
Moreover,
$$
K_0(\cla_\xi)/H
\cong
\bigl(\Z\oplus\Z_{n-1}\bigr)
\big/
\bigl(\{0\}\oplus\Z_{n-1}\bigr)
\cong
\Z.
$$
Therefore, there is a short exact sequence
$$
0
\longrightarrow
\Z
\longrightarrow
K_0(\cla_\xi)
\longrightarrow
\Z
\longrightarrow
0.
$$
Since $\Z$ is free, this sequence splits. Consequently,
$$
K_0(\cla_{\xi})
\cong
\Z^2,
\qquad
K_1(\cla_{\xi})
\cong
\Z.
$$
\end{proof}

\begin{example}\label{ex:n3-d2}
Let \(n=3\) and choose $\xi=e_1^{\otimes 2}\in\mathcal M^\perp.$
Under the unitary identification
$U_1:\mathcal M^\perp\longrightarrow H^2(\mathbb D),
\,
U_1(e_1^{\otimes p})=z^p,$
the symbol \(\xi\) corresponds to the finite Blaschke product
\[
\Theta(z)=U_1\xi=z^2.
\]
Thus \(\xi\) has finite Blaschke degree \(d=2\). By
Theorem~\ref{thm:index}, $\operatorname{ind}(W_\xi)=-2,$
and $I-W_\xi W_\xi^*$
is a rank-two projection. Theorem~\ref{thm:k-theory} gives
\[
K_1(\mathcal A_\xi)=0, \qquad
K_0(\mathcal A_\xi)
\cong
\mathbb Z\oplus\mathbb Z_{d(n-1)}
=
\mathbb Z\oplus\mathbb Z_4.
\] 
We can also see explicitly why the torsion summand is
\(\mathbb Z_4\), rather than \(\mathbb Z_2\oplus\mathbb Z_2\).
Let $x=[1_{\mathcal A_\xi}]_0 \in K_0(\mathcal A_\xi).$
Since $P_\Omega
=
I-\sum_{i=1}^3S_iS_i^*$
and each \(S_i\) is an isometry, we have
\[
i_*([P_\Omega]_0)
=
[1_{\mathcal A_\xi}]_0
-
\sum_{i=1}^3[S_iS_i^*]_0
=
-2x.
\]
The element \(i_*([P_\Omega]_0)\) has order \(d=2\), while the image
of \(x\) in $K_0(\mathcal Q_3)\cong\mathbb Z\oplus\mathbb Z_2$
is the class of the unit and has order \(2\). Therefore
\[
2x=-i_*([P_\Omega]_0)\neq0,
\qquad
4x=0.
\]
Hence \(x\) has order \(4\) and generates the torsion subgroup of
\(K_0(\mathcal A_\xi)\). Thus
\[
K_0(\mathcal A_\xi)
\cong
\mathbb Z\oplus\mathbb Z_4.
\]
\end{example}

\begin{corollary}\label{cor:extension-invariant}
Let \(\xi\) and \(\eta\) be finite Blaschke symbols of degrees
\(d_\xi\) and \(d_\eta\), respectively.
For $\alpha\in\{\xi,\eta\}$, let
$\partial_1^\alpha:
K_1(\mathcal Q(O_n))
\longrightarrow
K_0(\mathcal K(\mathcal F_n^2))
$
denote the connecting map associated with $\mathcal A_\alpha$. Under the identifications determined by
\[
[u]_1\longmapsto1,
\qquad
[P_\Omega]_0\longmapsto1,
\]
the map $\partial_1^\alpha$ is multiplication by $-d_\alpha$.
Consequently, if
\(d_\xi\neq d_\eta\), then the two extensions have different connecting
maps.
\end{corollary}

\begin{proof}
Following the boundary-map computation in the proof of
Theorem~\ref{thm:k-theory}, the canonical odometer unitary
$u\in\clq(O_n)$ represents the generator of $K_1(\clq(O_n))\cong \Z$, and
$W_\alpha$ is a Fredholm lift of $u$. Therefore
\[
\partial_1^\alpha([u]_1)
=
\operatorname{ind}(W_\alpha)
=
-d_\alpha
\]
by Theorem~\ref{thm:index}. Hence $\partial_1^\alpha$ is multiplication by
$-d_\alpha$.
\end{proof}

The resulting $K$-groups also distinguish the algebras arising from
symbols of different degrees.

\begin{corollary}\label{cor:non-isomorphism}
Let $\xi$ and $\eta$ be finite Blaschke symbols of degrees
$d_\xi,d_\eta\geq 0$, respectively. If $d_\xi\neq d_\eta$, then
$$
K_0(\cla_\xi)\not\cong K_0(\cla_\eta)
$$
as abelian groups. Consequently,
$
\cla_\xi\not\cong\cla_\eta
$
as $C^*$-algebras.
\end{corollary}

\begin{proof}
Suppose first that $d_\xi,d_\eta\geq 1$. By
Theorem~\ref{thm:k-theory},
$
K_0(\cla_\alpha)
\cong
\mathbb Z\oplus\mathbb Z_{d_\alpha(n-1)},
\,
\alpha\in\{\xi,\eta\}.
$
Hence the torsion subgroup of $K_0(\cla_\alpha)$ has order
$d_\alpha(n-1)$. Since $d_\xi\neq d_\eta$, these torsion subgroups
have different orders, and therefore
$$
K_0(\cla_\xi)\not\cong K_0(\cla_\eta).
$$
It remains to consider the case in which one of the degrees is zero.
Without loss of generality, assume that $d_\xi=0$ and $d_\eta\geq 1$.
By Proposition~\ref{prop:degree-zero},
$
K_0(\cla_\xi)\cong\mathbb Z^2,
$
whereas Theorem~\ref{thm:k-theory} gives
$
K_0(\cla_\eta)
\cong
\mathbb Z\oplus\mathbb Z_{d_\eta(n-1)}.
$
The former group has free rank $2$, while the latter has free rank $1$.
Thus these groups are not isomorphic.
Since $K$-theory is invariant under $*$-isomorphism, it follows that
$\cla_\xi$ and $\cla_\eta$ cannot be $*$-isomorphic.
\end{proof}

\begin{remark}
The preceding corollary distinguishes finite Blaschke symbols of different
degrees, but does not provide a complete classification of the algebras
${\cla_\xi}$. Indeed, symbols of the same degree have isomorphic pointed
$K$-groups. Distinguishing such
algebras may therefore require finer extension-theoretic information,
such as the Busby invariants of the essential extensions
\[
0
\longrightarrow
\clk(\clf_n^2)
\longrightarrow
\cla_\xi
\longrightarrow
\clq(O_n)
\longrightarrow
0.
\]
\end{remark}

\section*{Acknowledgements}
The author thanks the Department of Mathematics,
Technion--Israel Institute of Technology, Haifa, Israel, for its
institutional support and for providing a conducive research environment
during the preparation of this work.

\end{document}